\documentclass[draft,reqno,11pt]{amsart}

   \usepackage[hmargin=3cm, vmargin=5cm]{geometry}

\usepackage{amsfonts}
\usepackage{amssymb}
\usepackage{amsmath}
\usepackage{amsthm}
\usepackage{latexsym}
\usepackage{graphicx}
\usepackage{layout}
\usepackage{color}
\usepackage{geometry}

\newtheorem{theorem}{Theorem}[section]

\newtheorem{proposition}[theorem]{Proposition}
\newtheorem{lemma}[theorem]{Lemma}

\newtheorem{remark}[theorem]{Remark}

\renewcommand{\(}{\left(}
\renewcommand{\)}{\right)}
\renewcommand{\[}{\left[}
\renewcommand{\]}{\right]}

 \newcommand{\rr}{ \mathbb{R}}

\begin{document}
\title[Supercritical problems on manifolds]{Supercritical problems on manifolds}

\author{Angela Pistoia}
\address[Angela Pistoia] {Dipartimento SBAI, Universit\`{a} di Roma ``La Sapienza", via Antonio Scarpa 16, 00161 Roma, Italy}
\email{pistoia@sbai.uniroma1.it}

\author{Giusi Vaira}
\address[Giusi Vaira] {Dipartimento di Matematica "G.Castelnuovo", Universit\`{a} di Roma ``La Sapienza", P.le Aldo Moro 5, 00185 Roma, Italy}
\email{giusi.vaira@gmail.com}

\maketitle
\date{}
\begin{abstract}
Let $(M,g)$ be a $m$-dimensional compact Riemannian manifold  without boundary.   Assume $\kappa\in C^2(M)$ is such that  $-\Delta_g+\kappa$ is coercive.
We prove the existence of a solution to the supercritical problems
$$ -\Delta_gu+\kappa u= u^p,\ u>0\quad\hbox{in}\ (M,g)\quad\hbox{and}\quad-\Delta_gu+\kappa u=\lambda e^u\quad\hbox{in}\ (M,g)
$$
which concentrate s along a $(m-1)-$dimensional submanifold of $M$
as  $p\to\infty$ and $\lambda\to0$, respectively, under suitable symmetry assumptions on the manifold $M$.



\end{abstract}

\section{Introduction}
Let $(M,g)$ be a $m$-dimensional compact Riemannian manifold   without boundary.
We are interested in studying the following problems
\begin{equation}\label{le}
-\Delta_gu+\kappa u= u^p,\ u>0\quad\hbox{in}\ (M,g)\end{equation}
when the exponent $p$ is large enough and
 \begin{equation}\label{mf}
-\Delta_gu+\kappa u=\lambda e^u\quad\hbox{in}\ (M,g)\end{equation}
when the real parameter $\lambda$ is positive and small enough.
 Here $\kappa\in C^2(M)$ is such that the operator $-\Delta_g  +\kappa  $ is coercive.

\medskip

Let $2^*_m=\frac{2m}{m-2}$ if $m\ge 3$ be the critical exponent for the embeddings of the Riemannian Sobolev space $H^1_g\(M\)$ into Lebesgue spaces.  Set $2^*_2=+\infty.$  Due to the compactness of the embedding $H^1_g\(M\)\hookrightarrow L^{p+1}_g(M)$, the equation \eqref{le} in the subcritical case, i.e. $1<p<2^*_m-1,$   has  always a solution obtained by minimizing the quotient
$$\min\limits_{u\in H^1_g\(M\)\atop u\not=0}{\int\limits_M\(|\nabla_g u|^2+\kappa(x)u^2\)d\mu_g\over\(\int\limits_M |u|^{p+1}d\mu_g\)^{2\over p+1}}.$$
In the critical case, i.e. $p=2^*_m-1$ with $m\ge 3,$ existence of solutions is related to the position of the potential $\kappa $ with respect to the geometric potential $k_g:= {m-2\over 4(m-1)} R_g$, where $R_g$ is the scalar curvature of the manifold. If $\kappa\equiv \kappa_g$, then problem \eqref{le}  is referred to as  the   Yamabe equation and it has always a solution (see Aubin \cite{Aub}, Schoen \cite{Sch1}, Trudinger \cite{Tru}, and Yamabe \cite{Yam} for early references on the subject). When $\kappa< \kappa_g$ somewhere in $M$, existence of a solution is guaranteed by a minimization argument (see for example Aubin \cite{Aub}). The situation turns out to be more complicate when  $\kappa \ge \kappa_g$ because
blow-up phenomena can occur as pointed-out by Druet in \cite{Dru1,Dru2}. Some existence results  have been recently obtained   by Esposito-Pistoia-Vetois \cite{EPV} in a perturbative setting when the potential $\kappa$ is close to the geometrical one.

\noindent
In the supercritical regime, i.e. $p>2^*_m-1$, existence of   solutions is a delicate issue. As far as we know the only existence result has been obtained by
 Micheletti-Pistoia-Vetois  in \cite{MPV} when the potential $\kappa$ is far away from the geometric potential and the exponent $p$ is close to  the critical one.

\medskip
 Up to our knowledge, problem \eqref{mf} has  been   studied in an Euclidean setting, namely on a bounded domain of $\rr^m$ with Neumann boundary conditions. The $2-$dimensional case has been considered by Del Pino and Wei  in \cite{dw}, where the authors found solutions blowing-up at one or more points inside   or on the boundary of the domain as $\lambda\to0.$   The higher dimensional case  has been only  treated when the domain is a ball   
 and different type of radial solutions have been found by     Biler in \cite{b}  and Pistoia and Vaira in \cite{pv}.

\bigskip

In this paper, we prove the existence of a solution to \eqref{le} and \eqref{mf} which concentrates along a $(m-1)-$dimensional submanifold of $M$
as $p\to\infty$ and $\lambda\to0,$ respectively, when $(M,g)$ is a warped product manifold.

We recall the notion of warped product introduced by Bishop and O'Neill in \cite{BO}.
Let $(B,g_B)$ and $(N,g_N)$ be two riemannian manifolds of dimensions $b$ and $n ,$ respectively. Let   $f\in C^2(B),$ $f> 0$ be a differentiable function.  Consider the product (differentiable) manifold $B\times N$ with its projections
$\pi: B\times N\to B$ and $\eta: B\times N\to N.$ The warped product $M = B\times _f N$ is the manifold
$B\times N$ furnished with the riemannian structure such that
$$\|X\|^2=\|\pi_*(X)\|^2+f^2(\pi x)\|\eta_*(X)\|^2$$
for every tangent vector $X\in T_xM,$ whose associated metric is $g=g_B+f^2g_N.$ $f$ is called {\em warping function}.
For example, every surface of revolution (not crossing the axis of revolution) is
isometric to a warped product, with $B$ the generating curve, $N=S^1$ and $f(x)$ the distance from $x \in B$ to the axis of revolution.

It is not difficult to check that if $u\in C^2(B\times _f N)$   then
\begin{equation}\label{equ2}\Delta _g u=\Delta _{g_B} u +{n \over f}g_B\(\nabla^B f, \nabla ^B u\)+{1\over f^2}\Delta _{g_N} u.\end{equation}
Assume $\kappa$ is invariant with respect to $N,$ i.e. $\kappa(x,y)=\kappa(x)$ for any $(x,y)\in B\times N.$
If we look for solutions to \eqref{le} or \eqref{mf} which are invariant with respect to $N,$ i.e. $u(x,y)=v(x)$ then by \eqref{equ2}
we immediately deduce that $u$ solves \eqref{mf} or \eqref{le} if and only if $v$ solves
\begin{equation}\label{equ3}-\Delta _{g_B} v -{n \over f}g_B\(\nabla^B f, \nabla ^B v\)+\kappa v=v^p\ \hbox{or}\ \lambda e^v\quad\hbox{in}\ B.
\end{equation}

It is clear that if $v$ is a solution to problem \eqref{equ3} which concentrates at a point $x_0\in B$ then $u(x,y)=v(x)$ is a solution to problems \eqref{mf} or \eqref{le}
which concentrates along the fiber $\{x_0\}\times N$, which is a $n -$dimensional submanifold of $M.$
It is important to notice the the fiber $\{x_0\}\times N$ is   totally geodesic in $M$ (and in particular a minimal submanifold of $M$) if $x_0$ is a critical point of the warping function $f.$

\medskip
Here, we will consider the particular case when  $M=S^1 \times_f N,$ i.e.
  $B=S^1.$ In this case \eqref{equ3} turns out to be equivalent to the one dimensional periodic  boundary value problem
\begin{equation}\label{equ4}
\left\{
\begin{aligned}
&-  v'' - n  {f'(r)\over f(r)}v'+\kappa(r)v= v^p\ \hbox{or}\ \lambda e^v\quad \hbox{if}\ r\in(0,1)\\
&v(0)=v(1),\quad v'(0)=v'(1),
\end{aligned}\right.
\end{equation}
where $f\in   \mathfrak P^+$ and $\kappa\in \mathfrak P_c$ Here
\begin{eqnarray}\label{perio}
& \mathfrak P:=\left\{\theta\in C^2([0,1])\ :\ \theta(0)=\theta(1),\ \theta'(0)=\theta'(1),\ \theta^{''}(0)=\theta^{''}(1) \right\},\\ \label{perio2}
 & \mathfrak P^+:= \{\theta\in \mathfrak P\ :\ \min  _{[0,1]}\theta>0  \},\\ \label{perio3} 
& \mathfrak P_c:= \{\theta\in \mathfrak P\ :\ \hbox{the operator  $-v^{''}- n{f'\over f}v '+\kappa v$ is coercive} \}.\end{eqnarray}

It is clear that if $v$ is a solution to problem \eqref{equ4} which concentrates at a point $r_0\in (0,1)$ then $u(r,y)=v(r)$ is a solution to problems \eqref{le} or \eqref{mf}
which concentrates along the fiber $\{r_0\}\times N$, which is a $(m-1) -$dimensional submanifold of $M.$

In order to state our main results, it is useful to introduce the Green's function $G(r,s)$ of the operator
\begin{equation}\label{green}
 -v^{''}- n{f'(r)\over f(r)}v '+\kappa(r)v,\quad r\in(0,1)\end{equation}
with  periodic boundary conditions
$v(0)=v(1),\    v'(0)=v'(1).$
More precisely, the function
\begin{equation}\label{g1}
v(r):=\int\limits_0^1G(r,s)h(s)ds
\end{equation}
is the unique solution of the problem
$$\left\{
\begin{aligned}
&-v^{''}- n{f'(r)\over f(r)}v '+\kappa(r)v=h(r),\  r\in(0,1)\\
&v(0)=v(1),\quad   v'(0)=v'(1),
\end{aligned}\right.$$
for any smooth function $h.$

The Green's function can be decomposed as the sum of its singular part $\Gamma$ and its regular part $H$, namely
\begin{equation}\label{deco}
G(r,s)=\Gamma(r,s)+H(r,s),\end{equation}
where the function $\Gamma$ is defined as
\begin{equation}\label{gamma}
\Gamma(r,s):=\left\{
\begin{aligned}
&0 \ &  \hbox{if}\ 0\le r<s\le1,\\
&-f^n(s) \int\limits^r_s{1\over f^n(t)}dt\ & \hbox{if}\ 0\le s<r\le1,\\
\end{aligned}\right.\end{equation}
which solves
$$ -\partial_{rr}\Gamma(r,s)- n{f'(r)\over f(r)}\partial_r\Gamma
(r,s) =\delta_s(r),\  r\in(0,1)
 $$
in the sense of distributions. (Here $\delta_s$ is the Dirac function centered at $s$).
The regular part $H$  solves   the  periodic boundary value problem
\begin{equation}
\label{hn}\left\{
\begin{aligned}
&-\partial_{rr} H (r,s)- n{f'(r)\over f(r)}\partial_{r } H (r,s)+\kappa(r)H (r,s)=-\kappa(r)\Gamma (r,s),\  r,s\in[0,1],\\
&H  (0,s)=H  (1,s)-f^n(s) \int\limits^1_s{1\over f^n(t)}dt,\ s\in[0,1]\\
&\partial_{r } H  (0,s)=\partial_{r } H  (1,s)-{f^n(s)\over f^n(1)},\ s\in[0,1]. \\
\end{aligned}\right.\end{equation}
We point out that the Green's function, its regular part and its singular part depend on the warping function $f$ and the potential $\kappa.$ When it is necessary, we will focus the dependance of $H$ from  $f$ and $\kappa$ by writing $H_{f,\kappa}.$ In general, for sake of simplicity, we will omit it.

 Let us introduce the function
 \begin{equation}\label{A}
 \mathcal V_{f,\kappa}(r):={H_{f,\kappa}(r,r)\over f^n(r)},\quad r\in[0,1].
 \end{equation}
   Assume

   \begin{equation}\label{ipo}
   \hbox{\em there exists a non degenerate critical point $r_0\in(0,1)$ of the function
$\mathcal V _{f,\kappa} .$}
\end{equation}

In   Proposition  \ref{cor1} it will be shown that   the function $\mathcal V_{f,\kappa}$  has always at least a critical point, while in Section \ref{generi} we will prove that the non degeneracy condition,  is satisfied  for {\em  most}   warping functions $f$'s and for {\em most}  potential $\kappa$'s.

More precisely,  let $ \mathfrak P$ be the Banach space introduced in \eqref{perio} equipped with the norm
$\|\theta\| :=\sup\limits_{r\in[0,1]} \(|\theta(r)|+|\theta'(r)|+|\theta^{''}(r)|\).$

Let $\theta_0\in\mathfrak P$   and set
  $\mathfrak{B} (\theta_0,\rho):=\left\{ \theta\in  \mathfrak P   \ :\ \|\theta-\theta_0\| \le\rho\right\}.$
  We will prove the following generic result.
\begin{theorem}\label{main}
\begin{itemize}
\item[(i)] Let $\kappa\in  \mathfrak P_c$ be fixed.
For any $f_0\in  \mathfrak P^+$, the set
 $\mathfrak{A}:=\left\{f\in \mathfrak{B}(f_0,\rho)\ :\ \right.$  all the critical points of the  function
 $\mathcal V_{f,\kappa}$ are non degenerate $\left.\right\}$
 is a  dense subset of $\mathfrak{B}(f_0,\rho),$ provided $\rho$ is small enough.
 \item[(ii)] Let $f\in  \mathfrak P^+$ be fixed.
For any $\kappa_0\in  \mathfrak P_c$, the set
 $\mathfrak{A}:=\left\{\kappa\in \mathfrak{B}(\kappa_0,\rho)\ :\ \right.$  all the critical points of the  function
 $\mathcal V_{f,\kappa}$ are non degenerate $\left.\right\}$
 is a  dense  subset of $\mathfrak{B}(\kappa_0,\rho),$ provided $\rho$ is small enough.
 \end{itemize}
 \end{theorem}

Finally, we have all the ingredients to state our main results.
We will assume that
\begin{itemize}
\item $M=S^1\times_f N$ is a warped product manifold and  $f\in \mathfrak P^+$ (see \eqref{perio2}), 
\item $\kappa \in C^2(M)$ is invariant with respect to $N,$ i.e. $\kappa(r,y)=\kappa(r)$ for any $y\in N,$
and  $\kappa\in \mathfrak P_c$ (see \eqref{perio3}),
\item $r_0\in(0,1)$ is a non degenerate critical point of the function
$ \mathcal V_{f,\kappa}$ (see \eqref{A} and \eqref{ipo}).
\end{itemize}

\begin{theorem}\label{mainle}
There exists $p_0>0$ such that for any $p>p_0 $ problem
\begin{equation}\label{pp}
\left\{\begin{aligned}
&-v^{''}- n{f'(r)\over f(r)}v '+\kappa(r)v=v^p\quad r\in(0,1),\\
&v(r)>0\quad  r\in(0,1),\\
&v(0)=v(1),\quad v '(0)=v'(1).
\end{aligned}\right.
\end{equation}
 has a solution
$v_p$ such that
$$ v_p(r)\to{ G (r,r_0)\over H (r_0,r_0)}\quad\hbox{uniformly in $[0,1]$ as $p\to\infty$} .$$

In particular,   for ''most'' warping functions $f$'s  and for "most" potentials $\kappa$'s,
problem \eqref{le} has a solution invariant with respect to $N,$ i.e.
$u_p(r,y)=v_p(r) $  for any $y\in N,$ which concentrates along the  $(m-1)-$dimensional submanifold $\{r_0\}\times N$ of $M$ as $p\to\infty.$
\end{theorem}

\begin{theorem}\label{mainmf} There exists $\lambda_0>0$ such that for any $\lambda\in(0,\lambda_0)$ problem
\begin{equation}\label{pl}
\left\{\begin{aligned}
&-v^{''}- n{f'(r)\over f(r)}v '+\kappa(r)v=\lambda e^v\quad r\in(0,1),\\
&v(0)=v(1),\quad v'(0)=v'(1).
\end{aligned}\right.
\end{equation}has a solution
$v_\lambda$ such that
$$\epsilon _\lambda v_\lambda(r)\to2\sqrt 2 G (r,r_0)\quad\hbox{uniformly in $[0,1]$ as $\lambda\to0$}$$
for a suitable choice of positive numbers $\epsilon_\lambda $'s such that $\epsilon_\lambda\to0$ as $\lambda\to0.$

In particular,    for ''most'' warping functions $f$'s and for "most" potentials $\kappa$'s, problem
\eqref{mf} has a solution invariant with respect to $N,$ i.e.
$u_\lambda(r,y)=v_\lambda(r)$  for any $y\in N,$ which concentrates along the  $(m-1)-$dimensional submanifold $\{r_0\}\times N$ of $M$ as $\lambda\to0.$
 \end{theorem}

Let us make some remarks.

\begin{remark} If $(M,g)$ is a surface of revolution then we can find a solutions to problems \eqref{le} and \eqref{mf} which concentrate
along a curve which is the fiber $\{r_0\}\times S^1$. One could expect that the curve is a geodesic on the surface, but in general  this does not happen. Indeed, it would be a geodesic if $r_0$ was a critical point of the warping function $f$ (see \cite{BO}). In general $r_0$ is only a critical point of the  function
 $\mathcal V_{f,\kappa}$ defined in \eqref{A}.
\end{remark}

\begin{remark}
It would be interesting to find this kind of solutions in more general manifolds. For example, the easiest question which naturally arises is the following.
   Is it possible to build this kind of solutions if $(M,g)$ is a surface  whose genus  is not zero?
\end{remark}

 \bigskip

The paper is motivated by some recent results obtained by Grossi in \cite{G1} and \cite{G2} in the Euclidean case. More precisely, Grossi finds radial solutions to  the supercritical problems
$$
-\Delta u +\kappa(x)u=u^p\ \hbox{or} \ \lambda e^u\quad \hbox{in}\ \Omega, \qquad u=0\quad\hbox{on}\ \partial\Omega $$
if the exponent $p$ is large enough or the real parameter $\lambda$ is positive and small enough,
when $\Omega$ is the unit ball in the euclidean space $\rr^n $ and $\kappa$ is a  suitable positive radial function.
In particular, he is left to study the corresponding ODE problems
$$ -v^{''}- {n -1\over r}v '+\kappa(r)v=v^p\ \hbox{or}\ \lambda e^v  \quad \hbox{if}\ r\in(0,1),\qquad v'(0)=0,\ v(1)=0.$$
which are similar to \eqref{pl} and \eqref{pp}.
The proof given by Grossi relies on a fixed point argument and  the main tool is a clever use
of the Green's function associated to the operator
$-v^{''}- {n -1\over r}v '+\kappa(r)v$ with   boundary condition $v'(0)=0,\ v(1)=0.$

In this paper, we follow the same strategy developed by Grossi in \cite{G1} and \cite{G2}.
In Section \ref{progreen} we study the Green's function defined in \eqref{green} with periodic boundary condition and
in Section \ref{theproof} we sketch the proof of Theorem \ref{mainmf} and Theorem \ref{mainle}  which is obtained arguing exactly as in
\cite{G1} and \cite{G2}. The proof requires the crucial non degeneracy assumptions \eqref{ipo}, which is in general hard to check.
In Section \ref{generi} we prove Theorem \ref{main} which states  that assumption \eqref{ipo} is true for most warping functions $f$'s
and for most potentials $\kappa$'s.

\section{Properties of Green's function}\label{progreen}

For sake of simplicity, let us set $H:=H_{f,\kappa}.$

 \begin{lemma}\label{proH}
\begin{itemize}
\item[(i)] $H (r,s)$ is positive and it is uniformly bounded in $(0,1)$  with respect to $s,$
\item[(ii)] $G (r,s)f^n(r)=G (s,r)f^n(s)$  for any $r,s\in[0,1],$
\item[(iii)] $H (s,r)=H (r,s){f^n(r)\over f^n(s)}-f^n(r)\int\limits_s^r {1\over f^n(\sigma)}d\sigma$ for any $r,s\in[0,1],$
    \item[(iv)] $\partial_sH  (t,t)=\partial _rH (t,t)+nH (t,t){f'(t)\over f(t)}-1$ for  any $t\in[0,1].$
\end{itemize}
\end{lemma}

\begin{proof}

We set $a(r):=f^n(r).$ Therefore,
$a\in C^2\([0,1]\),$
  $a(0)=a(1),$ $a'(0)=a'(1)$ and $\min\limits_{r\in[0,1]}a(r)>0.$

{\em Proof of (i)}
Since $\Gamma (r,s)$ is negative and by maximum principle   $G (r,s)$ is positive, we immediately deduce that $H(r,s)>0$ for any $r,s\in[0,1].$
Let us prove that $H$ is bounded in $[0,1]\times[0,1].$
By \eqref{hn}  using the positivity of $G$, we deduce that
\begin{equation}\label{e1}
{\partial\over\partial r}\(a(r)\partial _r H(r,s)\)=a(r)\kappa(r)G(r,s)\ge0,\quad \forall\ r,s\in[0,1].
\end{equation}
By \eqref{e1} we get   estimate
$$a(1)\partial_r H(1,s)\ge a(r)\partial _r H(r,s)\ge a(0)\partial_r H(0,s),\quad \forall\ s\in[0,1],$$
which combined with the periodic boundary condition in \eqref{hn} gives
\begin{equation}\label{e2}c_1\ge \partial_r H(r,s)-{a(0)\over a(r)}\partial_r H(0,s)\ge 0,\quad \forall\ r,s\in[0,1],
\end{equation}
Here and in the follows, $c_i$'s denote    some positive constants.
We integrate \eqref{e2} over $[0,r]$ and $[r,1]$ and we get
$$cr\ge H(r,s)-H(0,s)-a(0)\partial_r H(0,s)\int\limits_0^r{1\over a(t)}dt\ge0,\quad \forall\ r,s\in[0,1]$$
and
$$c(1-r)\ge H(1,s)-H(r,s)-a(0)\partial_r H(0,s)\int\limits_r^1{1\over a(t)}dt\ge0,,\quad \forall\ r,s\in[0,1]$$
respectively.
Now, we add the previous two inequalities and we use the periodic boundary condition in \eqref{hn}, so we get
$$c\ge a(s)\int\limits_s^1{1\over a(t)}dt-a(0)\partial_r H(0,s)\int\limits_r^1{1\over a(t)}dt\ge0,,\quad \forall\ s\in[0,1],$$
which implies
$\max\limits_{s\in[0,1]}|\partial_r H(0,s)|\le c_2.$ By  the periodic boundary condition in \eqref{hn} we also
deduce  $\max\limits_{s\in[0,1]}|\partial_r H(1,s)|\le c_3.$ Then, by \eqref{e2} it follows
\begin{equation}\label{e3}\max\limits_{(r,s)\in[0,1]\times[0,1]}|\partial_r H(r,s)|\le c_4.
\end{equation}

Now,  we integrate equation \eqref{hn} over $[0,1]$ and we get
$$-a(1)\partial_r H(1,s) +a(0)\partial_r H(0,s)+\int\limits_0^1 a(r)\kappa(r)H(r,s)=-\int\limits_0^1a(r)\kappa(r)\Gamma (r,s)dr,$$
which implies, together with \eqref{e3}, that
\begin{equation}\label{e4}\int\limits_0^1 a(r)\kappa(r)H(r,s)dr\le c_5,\quad \forall\ s\in[0,1].\end{equation}
Moreover, if we multiply equation \eqref{hn} by $H(r,s)$ and integrate over $[0,1]$ we get
\begin{align*}&-a(1)\partial_r H(1,s)H(1,s) +a(0)\partial_r H(0,s)H(0,s)+\int\limits_0^1 a(r)\( \(\partial_r H\)^2(r,s)+\kappa(r)H^2(r,s)\)dr\\
&=-\int\limits_0^1a(r)\kappa(r)\Gamma (r,s)H(r,s)dr,\end{align*}
which implies together with \eqref{e4} and the periodic boundary condition in \eqref{hn}
$$\int\limits_0^1 a(r) \(\(\partial_r H\)^2(r,s)+\kappa(r)H^2(r,s)\)dr\le c_6+c_7H(1,s),\quad \forall\ s\in[0,1]. $$
Then we deduce
$$\|H(\cdot,s)\|_{W^{1,2}([0,1])}^2\le c_6+c_7\|H(\cdot,s)\|_{L^\infty([0,1])},\quad \forall\ s\in[0,1],$$
which implies together with the  continuous embedding $W^{1,2}([0,1])\hookrightarrow L^\infty([0,1])$ that
$$\|H(\cdot,s)\|_{L^\infty([0,1])}\le c_8,\quad \forall\ s\in[0,1],$$
namely
 $\max\limits_{(r,s)\in[0,1]\times[0,1]}H(r,s) \le c_8.$
That proves our claim.

{\em Proof of (ii)}
We know that
\begin{equation}\label{gi}\left\{
\begin{aligned}
&-{\partial\over\partial r}\({\partial G\over\partial r}(r,s)a(r)\)+\kappa(r)G(r,s)a(r)=\delta_s(r)a(r),\ r\in[0,1]\\
&G(0,s)=G(1,s),\quad{\partial G\over \partial r}(0,s)={\partial G\over \partial r}(1,s).
\end{aligned}\right.\end{equation}

Let $s,t\in[0,1].$ We multiply \eqref{gi} by $G(r,t)$ and integrate over $[0,1]$, so that
$$\int\limits_{0}^1a(r){\partial G\over\partial r}(r,s){\partial G\over\partial r}(r,t)dr+
\int\limits_{0}^1a(r)\kappa(r)G (r,s)G(r,t)dr=\int\limits_{0}^1\delta_s(r)a(r)G(r,t)dr=a(s)G(s,t).$$
Now,   we multiply \eqref{gi} with $s$ replaced by $t$ by $G(r,s)$ and integrate over $[0,1]$, so that
$$\int\limits_{0}^1a(r){\partial G\over\partial r}(r,s){\partial G\over\partial r}(r,t)dr+
\int\limits_{0}^1a(r)\kappa(r)G (r,s)G(r,t)dr=\int\limits_{0}^1\delta_t(r)a(r)G(r,s)dr=a(t)G(t,s).$$
The L.H.S.'s of the two equations are equal and so  {\em (ii)}  follows.

 {\em Proof of (iii) and (iv)}
 They follow by (ii) using the decomposition of $G.$

\end{proof}

The regularity of the regular part $H $ is studied in the following lemma.

\begin{lemma}
\label{regH}
\begin{itemize}
\item[(i)] $H \in C^2\([0,1]\times[0,1]\).$
\item[(ii)] $H \in C^3\([0,1]\times[0,1]\setminus D\),$ $D:=\{(r,r)\ |\ r\in[0,1]\}.$
\item[(iii)] There exist the limits
$$\partial_{rrr}^+H (\bar r,\bar r):=\lim\limits_{(r,s)\to(\bar r,\bar r)\atop r>s}\partial  _{rrr}H (r,s)\ \hbox{and}
\ \partial _{rrr}^-H (\bar r,\bar r):=\lim\limits_{(r,s)\to(\bar r,\bar r)\atop r<s}\partial  _{rrr}H (r,s).$$
The same happens for all the third order derivatives $\partial  _{rrs}H   ,$ $\partial  _{rsr}H ,$ $\partial  _{srr}H ,$ $\partial  _{rss}H ,$ $\partial  _{srs}H ,$ $\partial  _{ssr}H $ and $\partial  _{sss}H .$
\end{itemize}
\end{lemma}
\begin{proof}
(i) follows by standard regularity theory  applied to equation \eqref{hn}.

If we differentiate  \eqref{hn} by $r$ we deduce that
\begin{align*}
&-\partial  _{rrr}H (r,s)-{a'(r)\over a(r)}\partial  _{rr}H (r,s)-{a^{''}(r)a(r)-(a'(r))^2\over a^2(r)}\partial_r H(r,s)+\kappa'(r)H(r,s)+\kappa(r)\partial_rH(r,s)\\
&= 0\  \hbox{if}\ r<s, \quad  =-\kappa'(r)a(s)\int\limits^r_s{1\over a(t)}dt-\kappa(r){a(s)\over a(r)}\  \hbox{if}\ r>s. \end{align*}
Therefore, by standard regularity theory we immediately deduce  (ii).
We also get   the existence of the limits of the derivative   $ \partial  _{rrr}H $ as in (iii).
Let us prove claim (iii) also for the limits of the derivative $ \partial  _{sss}H .$ All the other derivatives
 can be managed in a similar way.
Let
$$\alpha(r,s):={a(s)\over 2a(1)}r^2+\(a(s)\int\limits^1_s{1\over a(t)}dt-{a(s)\over 2a(1)}\)r,\quad r,s\in[0,1].$$
It is immediate to see that $H$ and $\alpha$ satisfies the same periodic boundary condition as in \eqref{hn}.
Therefore, using the definition of Green's function in \eqref{g1} and  the fact that $H$ solves problem  \eqref{hn}, it is easy to check that
$$
H(r,s)=\int\limits_0^1\[ A(t,s)-\kappa(t)\Gamma(t,s)\]G(r,t)dt=\int\limits_0^1 A(t,s)G(r,t)dt-a(s)\int\limits_0^1 \kappa(t) G(r,t)\int\limits_s^t{1\over a(\sigma)}d\sigma dt,$$
where
$$A(r,s):=\partial  _{rr}\alpha (r,s)+{a'(r)\over a(r)}\partial  _{r}\alpha (r,s)-\kappa(r)\alpha (r,s).$$
Then
$$\partial_sH(r,s)= \int\limits_0^1 \partial_sA(t,s)G(r,t)dt-a'(s)\int\limits_s^1\kappa(t) G(r,t)\int\limits^t_s{1\over a(\sigma)}d\sigma dt+\int\limits_s^1\kappa(t) G(r,t)dt$$
and
$$\partial_{ss}H(r,s)= \int\limits_0^1 \partial_{ss}A(t,s)G(r,t)dt-a^{''}(s)\int\limits_s^1\kappa(t)G(r,t)\int\limits^t_s{1\over a(\sigma)}d\sigma dt+{a'(s)\over a(s)}\int\limits_s^1\kappa(t)G(r,t)dt.$$
We differentiate again by $s $ and taking into account that
 $$G_s(r,s)=\left\{
\begin{aligned}
&\partial _sH (r,s)\ &\hbox{if}\ r<s\\
&\partial_r H (r,s)+1-a'(s)\int\limits_s^r{1\over a(t)}dt\ &\hbox{if}\ r>s \\
\end{aligned}\right.$$
we deduce that $\partial_{sss}H $ is a continuous function in $[0,1]\times[0,1]\setminus D $ and also the existence of the limits as in (iii).

That concludes the proof.\end{proof}

\begin{proposition}\label{cor1}
Let $\mathcal V=\mathcal V_{f,\kappa}$ as in \eqref{A}.
\begin{itemize}
\item[(i)] $\mathcal V \in C^2([0,1]).$
\item[(ii)]
$\mathcal V '(r_0)=0\  \hbox{(i.e. $r_0\in(0,1)$ is a   critical point of $\mathcal V $)}$ if and only if $\partial _r H (r_0,r_0)={1\over2} .$

 \item[(iii)] There always exists a critical point $r_0\in(0,1)$ of $\mathcal V .$

\item[(iv)]  $\mathcal V^{''}(r_0)\not=0$ (i.e. $r_0$ is a non degenerate critical point of $\mathcal V$)  if and only if $\partial _{rr}H (r_0,r_0)+\partial_{rs}H (r_0,r_0)\not=0.$

    \end{itemize}\end{proposition}
\begin{proof}
(i) follows by Lemma \eqref{regH}, while (ii) and (iv) follow by a straightforward computation, using (iv) of Lemma \ref{proH}. Let us prove (iii).
By the periodic boundary condition in \eqref{hn} we deduce
$$H(0,0)=H(1,0)-f^n(0)\int\limits_0^1{1\over f^n(t)}dt\quad\hbox{and}\quad H(0,1)=H(1,1).$$
and by (iii) of Lemma \ref{proH} we deduce
$$H(0,1)=H(1,0)-f^n(0)\int\limits_0^1{1\over f^n(t)}dt.$$
 Combining the three relations we immediately get $H(0,0)=H(1,1)$ and so  $\mathcal V(0) =\mathcal V(1).$ Therefore, either $\mathcal V$ is constant on $[0,1]$ or it has a minimum or a maximum point in $(0,1).$

\end{proof}

\section{A generic result}\label{generi}

In this section we will prove    Theorem \ref{main} using an abstract transversality theorem previously used by Quinn \cite{Q}, Saut and Temam \cite{ST} and Uhlenbeck \cite{U}. 

First of all, we point out that by Proposition \ref{cor1}, we are led to consider the map (see \eqref{perio2})
\begin{equation}\label{10} \mathfrak V(r,f,\kappa)=\partial _t H_{f,\kappa}(t,r)_{|_{t=r}}-{1\over2},\ r\in[0,1],\ f\in\mathfrak P^+,\ \kappa\in\mathfrak P_c.
\end{equation}
Indeed, $\bar r$ is a critical point of $\mathcal V_{f,\kappa}$ if and only if  $\mathfrak V(r,f,\kappa)=0$ and it is non degenerate if and only if
$\partial _r \mathfrak V(r,f,\kappa)_{|_{r=\bar r}}\not=0.$

 We shall apply the following abstract transversality theorem to the map $F(r,f):=\mathfrak V(r,f,\kappa)$ in the case (i) when $\kappa$ is fixed
 and to the map $K(r,\kappa):=\mathfrak V(r,f,\kappa)$ in the case (ii) when $f$ is fixed  (see \cite{Q,ST,U}).

 \begin{theorem}\label{tran}
 Let $X,Y,Z$ be three Banach spaces and $U\subset X,$ $V\subset Y$ open subsets.
 Let $F:U\times V\to Z$ be a $C^\alpha-$map with $\alpha\ge1.$ Assume that

 \begin{itemize}
 \item[i)] for any $y\in V$, $F(\cdot,y):U\to Z$ is a Fredholm map of index $l$ with $l\le\alpha;$
 \item[ii)] $0$ is a regular value of $F$, i.e. the operator $F'(x_0,y_0):X\times Y\to Z$ is onto at any point $(x_0,y_0)$ such that $F(x_0,y_0)=0;$
 \item[iii)] the map  $\pi\circ i:F^{-1}(0)\to Y$ is $\sigma-$proper, i.e.  $F^{-1}(0)=\cup_{\eta=1}^{+\infty} C_\eta$
 where $C_\eta$ is a closed set and the restriction $\pi\circ i_{|_{C_\eta}}$ is proper for any $\eta$; here $i:F^{-1}(0)\to Y$ is the canonical embedding and $\pi:X\times Y\to Y$ is the projection.
 \end{itemize}

 Then the set
$\Theta:=\left\{y\in V\ :\ 0\ \hbox{is  a regular value of } F(\cdot,y)\right\}$
 is a  residual subset of $V$, i.e. $V\setminus \Theta$ is a countable union of closet subsets without interior points.

\end{theorem}

\medskip

\begin{proof}[Proof of Theorem \ref{main}]

$ $

\underline{\em Case (i)} We fix  $\kappa\in\mathfrak P_c$.
We are going to apply the transversality theorem \ref{tran} to the map $F$ defined by
$$F:(0,1)\times \mathfrak{P}^+  \to\rr,\quad F(r,f ):=\partial _t H_{f ,\kappa}(t,r)_{|_{t=r}}-{1\over2}.
$$

In this case we have
$X=Z=\rr ,$ $Y= \mathfrak P,$ $U=(0,1)\subset\rr $ and $V=\mathfrak{B}(f_0,\rho)\subset  \mathfrak P,$
where   $\rho$ is small enough.
By Lemma \ref{regH} and by standard regularity theory, we easily deduce that $F$ is a $C^2-$map.
Since $X=Z$ is a finite dimensional space, it is easy to check that for any $f\in \mathfrak{B}(f_0,\rho)$ the map $r\to F(r,r)$ is a Fredholm map of index $0$ and then assumption i) holds.
  As far as it concerns assumption  iii), we have that
$$F^{-1}(0)= \cup_{\eta=1}^{+\infty} C_\eta,\ \hbox{where}\ C_\eta:=\left\{ \[{1\over\eta},1-{1\over\eta}\]\times \overline  {\mathfrak{B}   \(f_0,\rho-{1\over \eta} \)} \right\}\cap F^{-1}(0) .$$
We can show  that the restriction
$\pi\circ i_{|_{C_\eta}}$ is proper, namely if the sequence $(f_n)\subset \overline  {\mathfrak{B}  \(f_0,\rho-{1\over \eta} \) }$ converges to $f$ and the sequence $(r_n)\subset  \[{1\over\eta},1-{1\over\eta}\]$ is such that
$F(r_n,f_n)=0$ then there exists a subsequence of $(r_n)$ which converges to $r\in \[{1\over\eta},1-{1\over\eta}\] $ and $F(r ,f)=0.$

Once assumption ii) is proved,   we can apply the transversality theorem \ref{tran} and we get that
the set
\begin{align*}
 \mathfrak{A}:=&\left\{f\in \mathfrak{B}(f_0,\rho) \ :\ D_rF(r,f)\not=0\ \hbox{  at any point}\ (r,f)\  \hbox{such that}\ F(r,f)=0\right\}\nonumber\\
=&\left\{f\in \mathfrak{B}(f_0,\rho) \ :\  \hbox{the critical points of $\mathcal V_{f,\kappa}$   are nondegenerate}
  \right\}\end{align*}
is a residual, and hence dense, subset of $\mathfrak{B}(f_0,\rho).$

{\em Proof of assumption ii)}

Let us fix $(\bar r,\bar f)\in (0,1)\times \mathfrak{B}(f_0,\rho)$ such that
\begin{equation}\label{f1}
F(\bar r,\bar f)=0,\ \hbox{i.e.}\ \partial_r H_{\bar f,\kappa}(\bar r,\bar r)={1\over2}.
\end{equation}
We have to prove that the map $D_fF(\bar r,\bar f):\mathfrak{P} \to\rr $  defined by
 $\theta\to D_f  \partial _t H_{f,\kappa}(t,\bar r)_{|_{t=\bar r\atop f=\bar f}}[\theta] $ is surjective, namely
    there exists $\theta\in\mathfrak{P}$ such that
$$D_fF(\bar r,\bar f)[\theta]=D_f  \partial _t H_{f,\kappa}(t,\bar r)_{|_{t=\bar r\atop f=\bar f}}[\theta]\not=0.$$
 We point out that
 $$D_f  \partial _t H_{f,\kappa}(t,\bar r)_{|_{t=\bar r\atop f=\bar f}}[\theta]=\partial _t D_f   H_{f,\kappa}(t,\bar r)[\theta ] _{|_{t=\bar r\atop f=\bar f}}$$
 Therefore, if we set $w_\theta(t):=D_f   H_ {f,\kappa}(t,\bar r)[\theta ](t)$ we have to prove that
 \begin{equation}
 \label{w0}
 \hbox{there exists $\theta\in\mathfrak{P}$ such that
 $w'_\theta (\bar r)\not=0.$}
 \end{equation}
 It is easy to check that the function $w_\theta$ solves the linear problem
 \begin{equation}\label{w1}
 \left\{
 \begin{aligned}
 &-w^{''}_\theta-n{\bar f'\over\bar f}w'_\theta+\kappa(t)w_\theta=n\left({\theta\over\bar f}\right)'\partial_t H_{\bar f,\kappa}(t,\bar r)+\kappa(t)\gamma(t),\quad t\in (0,1),\\
 &w_\theta(0)=w_\theta(1)-n\bar f^{n-1}(\bar r)\theta(\bar r)\int\limits_{\bar r}^1{1\over \bar f^n(s)}ds+n\bar f^{n }(\bar r)\int\limits_{\bar r}^1{\theta(s)\over \bar f^{n+1}(s)}ds\\
 &w'_\theta(0)=w'_\theta(1)-n{\bar f^{n-1}(\bar r)\over\bar f^{n }(1) }\theta(\bar r) +n{\bar f^{n }(\bar r)\over\bar f^{n+1}(1)} \theta(1). \\
 \end{aligned}
 \right.
 \end{equation}
where
$$\gamma(t):=0\ \hbox{ if $t<\bar r$ and}\ \gamma(t):=n\bar f^{n-1}(\bar r)\theta(\bar r)\int\limits_{\bar r}^t{1\over \bar f^n(s)}ds-
n\bar f^{n }(\bar r)\int\limits_{\bar r}^t{\theta(s)\over \bar f^{n+1}(s)}ds\ \hbox{ if $t\ge\bar r$.}$$
  Let $\chi\in C^2(\rr)$ be an even function such that $\chi'(r)\le 0$ if $r\ge0,$ $\chi(r)=0$ if $|r|\ge1$ and $\chi(0)=\max\limits_{ \rr} \chi=1.$
  For $\epsilon>0$ small we set
$\theta_\epsilon(t):=\chi_\epsilon(t)\bar f(t), $ with $
 \chi_\epsilon(t):=\chi\({t-\bar r \over\epsilon}\).$
  Therefore,  using the definition of $\Gamma$ in \eqref{gamma}, the function $w_\epsilon:=w_{\theta_\epsilon}$ solves \eqref{w1} which reads as
$$
 \left\{
 \begin{aligned}
 &-w^{''}_\epsilon-n{\bar f'\over\bar f}w'_\epsilon+\kappa(t)w_\epsilon=n \chi_\epsilon'(t)\partial_t H_{\bar f,\kappa}(t,\bar r)-n\kappa(t)\Gamma (t,\bar r)+\rho_\epsilon(t) ,\quad t\in (0,1),\\
 &w_\epsilon(0)= w_\epsilon(1)-n\bar f^{n}(\bar r)  \int\limits_{\bar r}^{\bar 1}{1\over \bar f^n(s)}ds+o(1)\\
 &w'_\epsilon(0)= w'_\epsilon(1)-n{\bar f^{n}(\bar r)\over\bar f^{n }(1) }+o(1), \\
 \end{aligned}
 \right.
$$
 where $\rho_\epsilon(t)\to0$ $C^0-$uniformly in $[0,1].$
 Then $w_\epsilon:=\tilde w_\epsilon+\hat w_\epsilon$  where $\tilde w_\epsilon$ solves
 \begin{equation}\label{w4}
 \left\{
 \begin{aligned}
 &-\tilde w^{''}_\epsilon-n{\bar f'\over\bar f}\tilde w'_\epsilon+\kappa(t)\tilde w_\epsilon=n \chi_\epsilon'(t)\partial_tH_{\bar f,\kappa}(t,\bar r) ,\quad t\in (0,1),\\
 &\tilde w_\epsilon(0)=\tilde w_\epsilon(1)  \\
 &\tilde w'_\epsilon(0)=\tilde w'_\epsilon(1)  \\
 \end{aligned}
 \right.
 \end{equation}
 and
 $\hat w_\epsilon$ solves
 \begin{equation}\label{w5}
 \left\{
 \begin{aligned}
 &-\hat w^{''}_\epsilon-n{\bar f'\over\bar f}\hat w  '_\epsilon+\kappa(t)\hat w _\epsilon=-n\kappa(t)\Gamma(t,\bar r) +\rho_\epsilon(t),\quad t\in (0,1),\\
 &\hat w _\epsilon(0)=\hat w_\epsilon(1)
 -n\bar f^{n}(\bar r)  \int\limits_{\bar r}^{\bar 1}{1\over \bar f^n(s)}ds  +o(1) \\
 &\hat w'_\epsilon(0)=\hat w'_\epsilon(1) -n{\bar f^{n}(\bar r)\over\bar f^{n }(1) }+o(1) \\
 \end{aligned}
 \right.
 \end{equation}
 Then, by \eqref{w4} using the definition of the Green's function $G$ given in \eqref{g1} we get
 $$\tilde w  _\epsilon(t)=\int\limits_0^1  n \chi_\epsilon'(s)\partial_tH_{\bar f,\kappa}(s,\bar r)  G (t,s)ds$$
 and so using the decomposition of $G $ given in \eqref{deco} we get
\begin{align}\label{v1}
&\tilde w'_\epsilon(\bar r)=\int\limits_0^1 n \chi_\epsilon'(s)\partial_tH_{\bar f,\kappa}(s,\bar r) \partial_r G  (\bar r,s)ds=\int\limits_{\bar r-\epsilon}^{\bar r+\epsilon} n {1\over\epsilon}\chi '\({s-\bar r\over\epsilon}\)\partial_tH_{\bar f,\kappa}(s,\bar r) \partial_r G (\bar r,s)ds\nonumber\\
&=\int\limits_{-1}^{1}  n {1\over\epsilon}\chi '\({\sigma}\)\partial_tH_{\bar f,\kappa}(\epsilon\sigma+\bar r,\bar r) \partial_r G (\bar r,\epsilon\sigma+\bar r)d\sigma\nonumber\\
&=\int\limits_{-1}^{0}  n  \chi '\({\sigma}\) \partial_tH_{\bar f,\kappa}(\epsilon\sigma+\bar r,\bar r)\partial_r H_{\bar f,\kappa}(\bar r,\epsilon\sigma+\bar r)  d\sigma\nonumber\\ &+ \int\limits_{0}^{1}   n  \chi '\({\sigma}\)\partial_tH_{\bar f,\kappa}(\epsilon\sigma+\bar r,\bar r) \(\partial_r H_{\bar f,\kappa}(\bar r,\epsilon\sigma+\bar r)-{\bar f^n(\epsilon\sigma+\bar r)\over f^n(\bar r)}\)d\sigma={n\over2}+o(1),
\end{align}
 because of \eqref{f1}.
Moreover, by \eqref{w5} using the standard regularity theory and the definition of $H_{\bar f,\kappa}$ given in \eqref{hn}, we immediately deduce  that $\hat w_\epsilon(r)\to nH_{\bar f}(r,\bar r)$ $C^1-$uniformly  in $[0,1].$ In particular,
\begin{equation}\label{z1}
\hat w'_\epsilon(\bar r)=n\partial_r H_{\bar f,\kappa}(\bar r,\bar r)+o(1)={n\over 2}+o(1),
 \end{equation}
 because of \eqref{f1}.
 Finally, by \eqref{v1} and \eqref{z1} we immediately deduce that
 $w'_\epsilon(\bar r)\not=0$ provided $\epsilon$ is small enough, which proves \eqref{w0}.

 \medskip
 $ $

\underline{\em Case (ii)}
We fix  $f\in\mathfrak P^+$.
We are going to apply the transversality theorem \ref{tran} to the map $K$ defined by
$$K:(0,1)\times \mathfrak{P}_c\to\rr,\quad K(r,\kappa):=\partial _t H_{f,\kappa }(t,r)_{|_{t=r}}-{1\over2}.
$$
We argue exactly as in the previous case.
Once  assumption ii) is proved,   we can apply the transversality theorem \ref{tran} and we get that
the set
\begin{align*}
 \mathfrak{A}:=&\left\{\kappa\in \mathfrak{B}(\kappa_0,\rho) \ :\ D_rK(r,\kappa)\not=0\ \hbox{  at any point}\ (r,\kappa)\  \hbox{such that}\ K(r,\kappa)=0\right\}\nonumber\\
=&\left\{\kappa\in \mathfrak{B}(\kappa_0,\rho) \ :\  \hbox{the critical points of $\mathcal V_{f,\kappa}$   are nondegenerate}
  \right\}\end{align*}
is a residual, and hence dense, subset of $\mathfrak{B}(\kappa_0,\rho).$

{\em Proof of assumption ii)}

Let us fix $(\bar r,\bar \kappa)\in (0,1)\times \mathfrak{B}(\kappa_0,\rho)$ such that
\begin{equation}\label{k1}
K(\bar r,\bar \kappa)=0,\ \hbox{i.e.}\ \partial_r H_{f,\bar \kappa}(\bar r,\bar r)={1\over2}.
\end{equation}
We have to prove that the map $D_\kappa K(\bar r,\bar \kappa):\mathfrak{P} \to\rr $  defined by
 $\theta\to D_\kappa  \partial _t H_{f,\kappa}(t,\bar r)_{|_{t=\bar r\atop \kappa=\bar\kappa}}[\theta] $ is surjective, namely
    there exists $\theta\in\mathfrak{P}$ such that
$$D_\kappa K(\bar r,\bar \kappa)[\theta]=D_\kappa  \partial _t H_{f,\kappa}(t,\bar r)_{|_{t=\bar r\atop \kappa=\bar \kappa}}[\theta]\not=0.$$
 We point out that
 $$D_\kappa  \partial _t H_{f,\kappa}(t,\bar r)_{|_{t=\bar r\atop \kappa=\bar \kappa}}[\theta]=\partial _t D_\kappa  H_{f,\kappa}(t,\bar r)[\theta ] _{|_{t=\bar r\atop \kappa=\bar\kappa}}$$
 Therefore, if we set $z_\theta(t):=D_\kappa   H_{f,\kappa}(t,\bar r)[\theta ](t)$ we have to prove that
 \begin{equation}
 \label{z0}
 \hbox{there exists $\theta\in\mathfrak{P}$ such that
 $z'_\theta (\bar r)\not=0.$}
 \end{equation}
 It is easy to check that the function $z_\theta$ solves the linear problem
 \begin{equation}\label{t1}
 \left\{
 \begin{aligned}
 &-z^{''}_\theta-n{ f'\over  f}z'_\theta+\bar \kappa(t)z_\theta=- \[H_{f,\bar \kappa}(t,\bar r)+\Gamma(t,\bar r)\]\theta ,\quad t\in (0,1),\\
 &z_\theta(0)=z_\theta(1),\  z'_\theta(0)=z'_\theta(1). \\
 \end{aligned}
 \right.
 \end{equation}
  Let $\zeta\in C^2(\rr)$ be a positive even function such that  $\zeta(r)=0$ if $|r|\ge1 .$
  For $\epsilon>0$ small we choose in \eqref{t1}
$\theta_\epsilon(t):= {1\over\epsilon}\zeta\({t-\bar r \over\epsilon}\). $
 Then,  using the definition of the Green's function $G $ given in \eqref{g1} we get
 $$z  _\epsilon(t)=-\int\limits_0^1   \theta_\epsilon (s) H_{f,\bar \kappa}(s,\bar r)  G (t,s)ds$$
 and so using the decomposition of $G $ given in \eqref{deco} we get
\begin{align*}
&z'_\epsilon(\bar r)=-\int\limits_0^1   \theta_\epsilon(s) \[H_{f,\bar \kappa}(s,\bar r) +\Gamma(s,\bar r)\]\partial_r G   (\bar r,s)ds\nonumber \\ & =-\int\limits_{-\bar r/\epsilon}^{(1-\bar r)/\epsilon}  \zeta({\sigma}) \[H_{f,\bar\kappa}(\epsilon\sigma+\bar r,\bar r) +\Gamma(\epsilon\sigma+\bar r,\bar r)\]\partial_r G (\bar r,\epsilon\sigma+\bar r)d\sigma\nonumber\\
&=-\int\limits_{-1}^{0}     \zeta\({\sigma}\)  H_{f,\bar \kappa}(\epsilon\sigma+\bar r,\bar r)\partial_r H_{f,\bar \kappa}(\bar r,\epsilon\sigma+\bar r)  d\sigma\nonumber\\ &- \int\limits_{0}^{1}    \zeta\({\sigma}\) \[H_{f,\bar \kappa}(\epsilon\sigma+\bar r,\bar r) -f^n(\bar r)\int\limits^{\epsilon\sigma+\bar r}_{\bar r}{1\over f^n(t)}dt\]\(\partial_r H_{f,\bar\kappa}(\bar r,\epsilon\sigma+\bar r)-{f^n(\epsilon\sigma+\bar r)\over f^n(\bar r)}\)d\sigma\nonumber \\
& =-{ 1\over2}\int\limits_0^1\zeta(\sigma)d\sigma +o(1),
\end{align*}
 because of \eqref{k1}.
 Then  $z'_\epsilon(\bar r)\not=0$ if $\epsilon$ is small enough, which proves \eqref{z0}.
 \end{proof}

\section{Proof of main results}\label{theproof}
Let us consider the limit problem
$$-U^{''}=e^{U}\ \hbox{in}\ \rr,
$$
whose solutions are
$$
U_{\epsilon,s}(r):=\ln{4\over\epsilon^2}{e^{\sqrt2{r-s\over\epsilon}}\over\(1+e^{\sqrt2{r-s\over\epsilon}}\)^2},\
r,s\in\rr,\ \epsilon>0.$$
Let $PU_{\epsilon,s}$ be the projection of $U_{\epsilon,s}$ on the interval $[0,1]$ with periodic boundary conditions,  namely $PU_{\epsilon,s}$ solves
the problem
\begin{equation}\label{proU}
\left\{\begin{aligned}
&-PU_{\epsilon,s}^{''}-n{f'(r)\over f(r)}PU_{\epsilon,s}'+\kappa(r)PU_{\epsilon,s}=e^{U_{\epsilon,s_\epsilon}(r)},\ r\in[0,1]\\
&PU_{\epsilon,s}(0)=PU_{\epsilon,s}(1),\quad PU_{\epsilon,s}'(0)=PU_{\epsilon,s}'(1),
\end{aligned}\right.\end{equation}

\begin{proof}[Proof of Theorem \ref{mainle}]
We   argue exactly as in \cite{G1}.
We look for a solution to problem \eqref{pp} as
$$v_p:=\rho\(w_\epsilon+\epsilon z_{1,\epsilon}+\epsilon^2 z_{2,\epsilon}\)+\phi_\epsilon$$
where
 $\rho=\rho(p)\to0$ as $p\to\infty.$
Here the first order term   $w_\epsilon=PU_{\epsilon,s_\epsilon}$ is defined in \eqref{proU}, 
the concentration point is $s_\epsilon:=r_0+ \epsilon\sigma+\epsilon^2\tau $ and the real numbers $\sigma$ and $\tau$ are chosen so that it is possible to find the higher order terms  $z_{1,\epsilon}$ and   $z_{2,\epsilon}$ as  solutions to  two different linear problems
\begin{equation}\label{zetal}
\left\{\begin{aligned}
&-z_{i,\epsilon}^{''}-n{f'(r)\over f(r)}z_{i,\epsilon}'+\kappa(r)z_{i,\epsilon}=e^{U_{\epsilon,s_\epsilon}(r)}h _{i,\epsilon},\ r\in[0,1]\\
&z_{i,\epsilon}(0)=z_{i,\epsilon}(1),\quad z_{i,\epsilon}'(0)=z_{i,\epsilon}'(1).
\end{aligned}\right.\end{equation}
for some suitable functions $h_{i,\epsilon} $ (see (4.9) and (4.23) in \cite{G1}). The functions $z_{1,\epsilon}$ and $z_{2,\epsilon}$ are built exactly as in Lemma 4.2 and Lemma 4.5 of \cite{G1}, respectively.
 The remainder term $\phi_\epsilon$ is found using a contraction mapping argument as in Section 7 of \cite{G2},
 once the parameter $\epsilon=\epsilon(p)\to0$ and $\rho=\rho(p)\to0$ as $p\to\infty$ are chosen in an appropriate way as in (5.7) and (5.8) of \cite{G1} (i.e. ${2\sqrt2\over\epsilon }H(r_0,r_0)\sim p $ and $\rho\sim{1\over p}$).

\end{proof}

\begin{proof}[Proof of Theorem \ref{mainmf}]
We   argue exactly as in \cite{G2}.
We look for a solution to problem \eqref{pl} as
$$v_\lambda:=w_\epsilon+\epsilon z_{1,\epsilon}+\epsilon^2 z_{2,\epsilon}+\phi_\epsilon$$
where the first order term    $w_\epsilon=PU_{\epsilon,s_\epsilon}$ is defined in \eqref{proU}, 
the concentration point is $s_\epsilon:=r_0+ \epsilon\sigma+\epsilon^2\tau $ and the real numbers $\sigma$ and $\tau$ are chosen so that it is possible to find the higher order terms  $z_{1,\epsilon}$ and   $z_{2,\epsilon}$ as  solutions to  two different linear problems like \eqref{zetal}
for some suitable functions $h_{i,\epsilon} $ (see (2.20) and (2.26) in \cite{G2}). The functions $z_{1,\epsilon}$ and $z_{2,\epsilon}$ are built exactly as in Lemma 2.4 and Lemma 2.6 of \cite{G2}, respectively.
 The remainder term $\phi_\epsilon$ is found using a contraction mapping argument as in Section 4 of \cite{G2},
 once the parameter $\epsilon=\epsilon(\lambda)\to0$ as $\lambda\to0$ is chosen in an appropriate way as in (3.7) of \cite{G2} (i.e. ${4\over\epsilon^2}\sim\lambda e^{{2\sqrt2\over\epsilon}H(r_0,r_0)} $).

\end{proof}

\end{document}